\documentclass[english,11pt,a4paper,leqno]{amsart}
\usepackage{amsmath,amstext, amsthm, amssymb}
\usepackage{latexsym, amsfonts, graphicx, xcolor}
\usepackage{mathrsfs}
\usepackage[latin1]{inputenc} 
\usepackage[T1]{fontenc} 
\usepackage{enumerate}
\usepackage{mathtools}
\usepackage[english]{babel} 
\usepackage{enumitem}
\usepackage{hyperref}

\makeatletter
\def\namedlabel#1#2{\begingroup
	#2%
	\def\@currentlabel{#2}%
	\phantomsection\label{#1}\endgroup
}

\makeatother

\hypersetup{
	colorlinks   = true, 
	urlcolor     = blue, 
	linkcolor    = blue, 
	citecolor   = red 
}

\numberwithin{equation}{section}

\newtheorem{thm}{Theorem}[section]
\newtheorem{prop}[thm]{Proposition}
\newtheorem{lem}[thm]{Lemma}
\newtheorem{coro}[thm]{Corollary} 

\newtheorem{conj}[thm]{Conjecture}
\newtheorem*{prop*}{Proposition}
\newtheorem*{thm*}{Theorem}

\theoremstyle{theorem}
\newtheorem*{warning*}{Warning}
\newtheorem*{rem*}{Remark}

\theoremstyle{definition} 

\newtheorem{rem}[thm]{Remark}

\newtheorem*{ex*}{Example}

\newtheorem*{nota*}{Notation}

\newcommand{\I}{\mathcal I}
\newcommand{\J}{\mathcal J}
\newcommand{\Q}{\overline{\mathbb Q}}
\newcommand{\C}{\mathbb{C}}
\newcommand{\Z}{\mathbb{Z}}

\newcommand{\K}{\mathbb{K}}

\newcommand{\OK}{{\mathbb Z_{\mathbb K}}}

\newcommand{\X}{\boldsymbol{X}}

\newcommand{\omeg}{{\boldsymbol \omega}}

\newcommand{\bbeta}{\boldsymbol \beta}

\newcommand{\p}{{\mathfrak{p}}}
\newcommand{\hgt}{{\rm ht}}

\newcommand{\val}{\operatorname{val}_{z}}

\title[A Liouville-type inequality for values of $M$-functions]{A  Liouville-type inequality for values of Mahler $M$-functions}

\author{Boris Adamczewski}
\address{
	Univ Lyon, Universit\'e Claude Bernard Lyon 1\\
	CNRS UMR 5208, Institut Camille Jordan, F-69622 Villeurbanne Cedex, France, and
Centre International de Rencontres Math\'ematiques (CIRM), 163 Av.\ de Luminy, 13009
Marseille, France}
\email{Boris.Adamczewski@math.cnrs.fr}

\author{Colin Faverjon}
\address{
	Univ Lyon, Universit\'e Claude Bernard Lyon 1\\
	CNRS UMR 5208, Institut Camille Jordan \\
	F-69622 Villeurbanne Cedex, France}
\email{colin.faverjon@math.cnrs.fr}
\date{\today}

\date{}
\thanks{}
\begin{document}
\begin{abstract}  
We establish a Liouville-type inequality for the values, at a common nonzero algebraic point, of arbitrary Mahler $M_q$-functions. As an application, we prove that no such value is a Liouville number, or even a $U$-number. This solves a long-standing problem in the field.  
 \end{abstract}

	\bibliographystyle{abbvr}
	\maketitle

\section{Introduction}\label{sec:introduction}

Let $q\ge 2$ be an integer, and let $\Q\subset\C$ denote the field of algebraic numbers. An $M_q$-function is a power series
$f(z)\in \Q[[z]]$
satisfying a functional equation of the form
\begin{equation}\label{eq:def:M-func}
a_0(z)f(z)+a_1(z)f(z^q)+\cdots+a_m(z)f(z^{q^m})=0\,,
\end{equation}
where $a_0,\ldots,a_m\in \Q[z]$ and $a_0a_m\neq 0$. We call $f$ an \emph{$M$-function} if it is an $M_q$-function for some integer $q\ge 2$.

Typical examples of $M$-functions are
\[
\sum_{n=0}^{\infty} z^{3^n},\qquad
\prod_{n=0}^\infty \frac{1}{1-z^{5^n}},\qquad
\sum_{n=0}^\infty \lfloor \log_2 n \rfloor z^n.
\]
Most importantly, generating series of automatic and regular sequences provide a large supply of examples; see, for instance, \cite{ABS}. 
For background on automatic sequences and their applications, we refer the reader to \cite{AS03}. 
An $M$-function is always meromorphic in the open unit disk, and the unit circle is a natural boundary unless the function is rational; see \cite{Ra92,BCR}.    
This explains why the assumption
$
0<|\alpha|<1$ 
appears implicitly throughout our results.

In this paper, we study Diophantine properties of the so-called $M$-values, that is, of elements of the set
\[
\mathbf M
:=
\{f(\alpha)\,:\, \alpha\in\Q,\ f \text{ is an } M\text{-function well-defined at } \alpha\}\, .
\]
Recall that a \emph{Liouville number} is an irrational real number $\xi$ such that, for every integer $w\ge 1$, there exists a rational number $p/q$ with
\[
\left|\xi-\frac{p}{q}\right|\le \frac{1}{q^w}\,\cdot
\]

Our initial motivation was to prove the following statement.

\begin{thm}{\, }
\label{thm:Lnumbers}
No $\xi\in \mathbf M$ is a Liouville number.
\end{thm}

The question addressed in Theorem~\ref{thm:Lnumbers} has a long history, which we briefly recalled in Section~\ref{sec:introbis}. Let us simply mention here that, in his first paper on the subject, published in 1929, Mahler \cite{Ma29} already observed, although without proof, that some of the numbers he had just proved to be transcendental should not be Liouville numbers. The corresponding statement for general elements of $\mathbf M$ seems to belong to the folklore of the subject, although it is difficult to trace a precise source.

Liouville's inequality was not only used, in its original form, to construct Liouville numbers and thereby provide the first examples of transcendental numbers, but also became, in its more general form, a fundamental and ubiquitous tool in transcendental number theory.   It asserts that, given 
$\alpha_1,\ldots,\alpha_m\in\Q$, there exists a positive real number $C_1$ such that, for every polynomial $P\in\Z[X_1,\ldots,X_m]$, one has
\[
\text{either } P(\alpha_1,\ldots,\alpha_m)=0 \quad
\text{or } \quad  |P(\alpha_1,\ldots,\alpha_m)|\ge H(P)^{-C_1}C_2(\deg(P)),
\]
where $C_2(d):=e^{-C_1d}$ and $C_1$ depends only on $\alpha_1,\ldots,\alpha_m$ (see, for instance, \cite[p.~84--85]{Wa_Liv}).  
Here $\deg(P)$ denotes the total degree of $P$, and $H(P)$ its height, that is, the maximum of the absolute values of its coefficients.

Theorem~\ref{thm:Lnumbers} is a straightforward consequence of the following Liouville-type inequality, in which the field $\Q$ is replaced by the ring
\begin{equation}\label{eq:mqalpha}
\mathbf M_{q,\alpha}
:=
\{f(\alpha)\,:\, f \text{ is an } M_q\text{-function well-defined at } \alpha\}\,, 
\end{equation}
where the parameters $q\geq 2$ and $\alpha\in\Q$ are fixed. 
This ring contains $\Q$ and its field of fractions has infinite transcendence degree.

\begin{thm}
\label{thm:Liouville}
Let $\xi_1,\ldots,\xi_m\in \mathbf M_{q,\alpha}$. Then there exist two positive real numbers $C_1$ and $\tau$ such that, for every polynomial $P\in\Z[X_1,\ldots,X_m]$, one has the following alternative:
\[
\text{either } P(\xi_1,\ldots,\xi_m)=0\,,
\]
or
\[
|P(\xi_1,\ldots,\xi_m)|
\ge
H(P)^{-C_1\deg(P)^\tau}C_2(\deg(P))\,,
\]
where 
$
C_2(d):=e^{-C_1d^{2\tau+2}}$, 
and where $C_1$ and $\tau$ depend only on $\xi_1,\ldots,\xi_m$.
\end{thm}

Theorem~\ref{thm:Liouville} implies the more general statement that $\mathbf M$ contains no $U$-number in the sense of Mahler's classification; see Section~\ref{sec: mahler}, Corollary~\ref{coro:U_numbers}. Theorem~\ref{thm:Liouville} will be proved using tools from elimination theory, first introduced in this framework by Nesterenko~\cite{Ne85} and later used by Becker~\cite{Bec91} and Nishioka~\cite{Ni91}. The new input, as compared with previous results of Becker and Nishioka, lies in the role played by the ideal of homogeneous relations over $\Q(z)$ among the underlying $M_q$-functions, together with the existence of \emph{good Mahler operators}. Most of the results proved in this paper were already contained in the unpublished preprint~\cite{AF25}.

The paper is organized as follows. In Section~\ref{sec:introbis}, we briefly review the state of the art surrounding Theorem~\ref{thm:Lnumbers} and state a more precise form of Theorem~\ref{thm:Liouville}, namely Theorem~\ref{thm:measure_sans_system}. We also state the more restrictive Theorem~\ref{thm:measure_system}, which concerns $M_q$-functions related by a linear Mahler system and evaluated at regular points, and is one of the key ingredients in the proof of Theorem~\ref{thm:Liouville}. In Section~\ref{sec:M_func}, we recall the necessary background on elimination theory and prove Theorem~\ref{thm:measure_system}. We then deduce Theorem~\ref{thm:measure_sans_system} in Section~\ref{sec:remove_sing}. In Section~\ref{sec: mahler}, we recall Mahler's classification and prove Corollary~\ref{coro:U_numbers}. Finally, in Section~\ref{sec: last}, we give an application \emph{\`a la Liouville} of our main result, showing that sufficiently lacunary series evaluated at $M$-values are transcendental.


\section{Statement of Theorems~\ref{thm:measure_system} and \ref{thm:measure_sans_system}}\label{sec:introbis}

Let us briefly review the state of the art surrounding Theorems~\ref{thm:Lnumbers} and \ref{thm:Liouville}. 
After the pioneering work of Mahler~\cite{Ma29}, several special cases of Theorem~\ref{thm:Lnumbers} were later established explicitly. In 1980, Galoch\-kin~\cite{Gal80} proved the result for values of $M$-functions satisfying a homogeneous or inhomogeneous Mahler equation of order~$1$. In 1993, Shallit conjectured that no automatic irrational real number is a Liouville number; this conjecture later appeared in print in \cite{Sh99}. Adamczewski and Cassaigne  \cite{AC06} proved Shallit's conjecture in 2006, using an approach inspired by  \cite{ABL,AB07} rather than Mahler's method. Finally, in 2015, Bell, Bugeaud, and Coons extended this result to values at points of the form $1/b$ of arbitrary Mahler functions with rational coefficients.

On the other hand, techniques from elimination theory, introduced in this area by Nesterenko \cite{Ne85}, have proved to be a very powerful tool for obtaining strong measures of algebraic independence for values of $M$-functions.  
Assume that the following conditions are satisfied: 
\begin{enumerate}[label=(\roman*)]

\item \label{Assump_i}
The vector
$
(f_1=f,\ldots,f_m)\in \Q[[z]]^m
$ 
is a solution of a linear Mahler system of the form
\begin{equation}\label{eq:mahlersystem}
Y(z^q)=A(z)Y(z),
\qquad
A(z)\in {\rm GL}_m(\Q(z)).
\end{equation}

\medskip

\item \label{Assump_ii}
The functions $f_1,\ldots,f_m$ are algebraically independent over $\Q(z)$.

\medskip

\item \label{Assump_iii}
The point $\alpha$ is a \emph{regular point} for  \eqref{eq:mahlersystem} meaning that
$A(z)$ is defined and invertible at $\alpha^{q^k}$ for every integer $k\ge 0$.

\end{enumerate}

\medskip

\noindent Then there exist a positive real number $C_1$ and, for every integer $d$, a positive real number $C_2(d)$ such that
\[
|P(f_1(\alpha),\ldots,f_m(\alpha))|
\ge
H(P)^{-C_1\deg(P)^m}C_2(\deg(P))
\]
for every nonzero polynomial $P\in \Z[X_1,\ldots,X_m]$. Here $\deg(P)$ denotes the total degree of $P$, and $H(P)$ its height, that is, the maximum of the absolute values of its coefficients.

Such a bound immediately implies that $f(\alpha)$ is not a Liouville number. 
Results of this type were first obtained by Becker~\cite{Bec91} and later refined by Nishioka~\cite[Thm.~1]{Ni91}, where she established that one may take
\[
C_2(d)=e^{-C_1d^{2m+2}}.
\]
In principle, the constant $C_1$ can be made effective, although no explicit bound seems to have been worked out in the literature. 

More generally, algebraic independence measures have been established in broader settings, for instance for certain $M$-functions that do not necessarily satisfy Assumption~\ref{Assump_ii}. In this direction, Philippon~\cite[Thm.~6]{Ph98} obtained an algebraic independence measure for linear Mahler systems with polynomial coefficients. In 2013, Zorin~\cite{Zo13} released a preprint on certain generalizations of $M$-functions and on algebraic independence measures for their values. In the original version, the main result concerning $M$-functions was incorrect. After a discussion with the first author, Zorin released revised versions of his preprint in 2016 and 2017, in which he announced, in particular, Corollary~\ref{coro:U_numbers}. His general result also implies an algebraic independence measure similar to Theorem~\ref{thm:measure_system}, although weaker than the one obtained here. However, this revised version was never published in a refereed journal, so the status of its proof remains unclear.
 In her PhD thesis, Fernandes~\cite{Fe19} proposed a result in the spirit of Theorem~\ref{thm:measure_system}, with weaker quantitative bounds, relying on an unpublished result of Jadot~\cite{Ja96}.

In order to prove Theorem~\ref{thm:Lnumbers} and Theorem~\ref{thm:Liouville}, we follow the approach based on elimination theory. Our first  result is a generalization of Nishioka's theorem in which condition~\ref{Assump_ii} is removed. The trade-off is that one can no longer exclude the possibility that
\[
P(f_1(\alpha),\ldots,f_m(\alpha))=0\,.
\] 
Accordingly, the result takes the form of an alternative, such as in  Liouville's inequality and Theorem~\ref{thm:Liouville}. 

\begin{thm}
\label{thm:measure_system}
Let $q\ge 2$, $m\ge 1$ be  integers, and let $f_1,\ldots,f_m\in\Q[[z]]$ be related by a Mahler system
\[
Y(z^q)=A(z)Y(z)\,,
\qquad
A(z)\in {\rm GL}_m(\Q(z))\,.
\]
Let $\alpha\in\Q$ with $0<|\alpha|<1$, and assume that $\alpha$ is regular for this system. Then there exists a positive real number $C_1$ such that, for every polynomial $P\in\Z[X_1,\ldots,X_m]$, one has the following alternative:
\[
\text{either } P(f_1(\alpha),\ldots,f_m(\alpha))=0,
\]
or
\[
|P(f_1(\alpha),\ldots,f_m(\alpha))|
\ge
H(P)^{-C_1\deg(P)^t}C_2(\deg(P))\,,
\]
where
\[
t:={\rm tr.deg}_{\Q(z)}(f_1,\ldots,f_m)
\qquad\text{and}\qquad
C_2(d):=e^{-C_1d^{2t+2}}.
\]
\end{thm}

\begin{rem}
Theorem~\ref{thm:measure_system} has the following consequence: if $i_1,\ldots,i_r\in\{1,\ldots,m\}$ are such that
$
f_{i_1}(\alpha),\ldots,f_{i_r}(\alpha)$ 
are algebraically independent, then
\[
|P(f_{i_1}(\alpha),\ldots,f_{i_r}(\alpha))|
\ge
H(P)^{-C_1\deg(P)^t}C_2(\deg(P))
\]
for every polynomial $P\in\Z[X_1,\ldots,X_r]$. 
 Moreover, if $t=m$, then the non-quantitative version of Nishioka's theorem \cite[Thm.~4.2.1]{Ni_Liv} implies that $f_1(\alpha),\ldots,f_m(\alpha)$ are algebraically independent.
In this case, Theorem~\ref{thm:measure_system} recovers the quantitative version of Nishioka's theorem \cite[Thm.~1]{Ni91} as a special case. 
 Recently, the authors  \cite{AF23,AF24_Annals} gave a new proof of \cite[Thm.~4.2.1]{Ni_Liv}  by a different method. That approach should also yield an analogue of Theorem~\ref{thm:measure_system}, albeit with weaker quantitative bounds.  
\end{rem}

Theorem~\ref{thm:measure_system} alone is not sufficient to derive Theorem~\ref{thm:Lnumbers}, since it requires the point 
$\alpha$ to be regular for the underlying Mahler system. The key to overcoming this difficulty lies in the existence of suitable \emph{good operators} attached to arbitrary $M$-functions. More precisely, all we need is that, for every algebraic point $\alpha$ and every $M$-function $f$ that is well-defined at $\alpha$, there exists a Mahler operator for which $f$ is a solution and for which $\alpha$ is a regular point. The existence of such operators was established recently by the authors in \cite{AF24_EM}. We shall deduce from this the following result, from which Theorem~\ref{thm:Liouville} follows immediately.

\begin{thm}
	\label{thm:measure_sans_system} 
	 Let $r\geq 1$ be an integer and $f_1,\ldots,f_r \in \Q[[z]]$ be $M_q$-functions for some $q\in \mathbb N_{\geq 2}$. Let $\alpha \in \Q$,  $0<\vert \alpha \vert < 1$, such that  $f_1,\ldots,f_r$ are well-defined at $\alpha$. Set
	$$\tau = {\rm tr.deg}_{\Q(z)}(f_i(z^{q^\ell}) \,:\, 1\leq i\leq r,\; \ell \geq 0)\,.$$
Then, there exists a positive real number $C_1$ such that the 
following alternative holds  for all polynomials $P\in\mathbb Z[X_1,\ldots,X_r]$:  
	\begin{eqnarray*}	
	\text{ either } & P(f_1(\alpha),\ldots,f_r(\alpha))=0 \\
\text{ or } & \vert P(f_1(\alpha),\ldots,f_r(\alpha)) \vert \geq H(P)^{-C_1\deg(P)^\tau}C_2(\deg(P)) \,.
	\end{eqnarray*}  
where $C_2(d)=e^{-C_1d^{2\tau+2}}$.
\end{thm}
\begin{rem}\label{rem:tau}
 If for every $i$, $1\leq i\leq r$, we let  $m_i$ denote the order of the minimal linear Mahler equation satisfied by $f_i$, then we have the inequality 
\begin{equation*}
\tau \leq m_1+\cdots +m_r \,.
\end{equation*} 
\end{rem}


\section{Proof of Theorem~\ref{thm:measure_system}}\label{sec:M_func}

To establish Theorem~\ref{thm:measure_system}, we follow the approach of Nesterenko~\cite{Ne85}, Becker~\cite{Bec91}, and Nishioka~\cite{Ni91}, which is based on elimination theory. For the reader's convenience, we adhere closely to the exposition in Nishioka's book \cite[Chap.~4]{Ni_Liv}.

By convention, we set $\log(0) = -\infty$. With this convention, all equalities and inequalities involving logarithms of non-negative real numbers in this section remain valid even when the argument of the logarithm is zero.

\subsection{Basics of Elimination Theory}

Given a prime ideal $\I$, we let $\hgt(\I)$ denote its height, that is, the largest integer $h$ for which there exist prime ideals 
$\mathfrak p_0,\ldots,\mathfrak p_h$ satisfying
\[
\mathfrak p_0 \subsetneq \mathfrak p_1 \subsetneq \cdots \subsetneq \mathfrak p_h = \I \,.
\]
The height of a possibly non-prime ideal $\I$ is defined by
\[
\hgt(\I):=\min\{\hgt(\mathfrak p) : \mathfrak p \text{ prime and } \I\subseteq \mathfrak p\}\,.
\]

Let $\X=(X_0,\ldots,X_m)$ be a tuple of indeterminates. For an ideal
$\I\subset \Z[\X]$, we denote by
\[
\I_{\mathbb Q}:=\I\,\mathbb Q[\X]\simeq \I\otimes_{\Z}\mathbb Q
\]
its extension to $\Q[\X]$. We shall be interested in homogeneous ideals
$\I\subset \Z[\X]$ such that $\I_{\mathbb Q}$ is a proper ideal of $\mathbb Q[\X]$. By extension, we say that $\I\subset \Z[\X]$ is proper if 
$\I_\mathbb Q$ is proper. 
Equivalently, this means that
\[
\I\cap \Z=(0)\,.
\]
For such an ideal $\I$, we let
\[
V(\I):=\{\boldsymbol{\xi}\in \mathbb P^m(\Q)\,:\, P(\boldsymbol{\xi})=0
\text{ for every } P\in \I\}
\]
denote the associated projective variety.

Any proper homogeneous ideal $\I \subset \Z[\X]$ admits a primary decomposition
\[
\I=\I_1\cap \cdots \cap \I_s\cap\I_{s+1}\cap \cdots \cap \I_u\, ,
\]
where $\I_1,\ldots,\I_s$ are proper, whereas $\I_{s+1},\ldots,\I_u$ are not. For $1\le i\le s$, we set $\mathfrak p_i=\sqrt{\I_i}$. The ideals $\mathfrak p_1,\ldots,\mathfrak p_s$ are called the \textit{associated proper prime ideals} of $\I$. 
We say that $\I$ is \emph{unmixed} if
\[
\hgt(\mathfrak p_i)=\hgt(\I)\qquad \text{for every } i\in\{1,\ldots,u\}\,.
\]
In particular, the varieties $V(\mathfrak p_1),\ldots,V(\mathfrak p_s)$ all have the same dimension.

Let $\I \subset \Z[\X]$ be a proper unmixed homogeneous ideal and set $r=m+1-\hgt(\I)$. Consider an $r\times (m+1)$ matrix of indeterminates $U_r=(u_{i,j})_{1\leq i \leq r,\,0\leq j \leq m}$ and define
\[
Y_i = \sum_{j=0}^m u_{i,j}X_j,\qquad 1\leq i \leq r.
\]
Let $\overline{\I}(r)$ denote the set of polynomials $P \in \Z[U_r]$ such that
\[
P X_j^n \in (\I,Y_1,\ldots,Y_r)\Z[U_r,\X],\qquad \forall j \in \{0,\ldots,m\}\, ,
\]
for some integer $n$. With this choice of $r$, the ideal $\overline{\I}(r)$ is principal \cite[Chap.~3, Prop.~4.4]{NP01}.

To each proper unmixed homogeneous ideal $\I$, we associate the following quantities. Let $r=m+1-\hgt(\I)$ and let $F\in \Z[U_r]$ 
be such that $\overline{\I}(r)=(F)$.
We define\footnote{In \cite{Ni_Liv}, the quantity $\deg(\I)$ is denoted by $N(\I)$.}
\[
\deg(\I):=\deg_{U_r} F
\]
and
\[
H(\I):=H(F)\,,
\]
where $H(F)$ denotes the maximum of the absolute values of the coefficients of $F$. 
For a principal homogeneous ideal $\I=(P)$, we have \cite[Lem.~4.1.1]{Ni_Liv}
\begin{equation}\label{eq:deg_H_principal}
\deg(\I) = \deg(P)\qquad \mbox{and} \qquad
\log H(\I) \leq \log H(P)+m^2\deg(P).
\end{equation}

Let $\omeg=(\omega_0,\ldots,\omega_m) \in \C^{m+1}\setminus\{0\}$. The distance\footnote{In \cite{Ni_Liv}, this quantity is denoted by $\Vert \I,\omeg\Vert$.}
 from $\omeg$ to $V(\I)$ is defined by
\[
{\rm dist}(\omeg,\I)
=
\min_{\bbeta=(\beta_0,\ldots,\beta_m) \in V(\I)}
\left(
|\bbeta|^{-1} |\omeg|^{-1}
\max_{i<j} |\omega_i\beta_j - \omega_j\beta_i|
\right)\,,
\]
where $|(x_0,\ldots,x_m)|=\max_i |x_i|$ for $(x_0,\ldots,x_m)\in\mathbb C^m$. 
Closely related to ${\rm dist}(\omeg,\I)$ is the \textit{absolute value of $\I$ at $\omeg$}, denoted by $|\I(\omeg)|$. A precise definition can be found in \cite[p.~119]{Ni_Liv}. We only recall the properties that will be used. 
If $\I=(P)$ with $P$ homogeneous, then \cite[Lem.~4.1.1]{Ni_Liv}
\begin{equation}\label{eq:absolute_value_principale}
|\I(\omeg)|
\leq
|P(\omeg)|\, |\omeg|^{-\deg(P)} (m+1)^{2m\deg(P)}\,.
\end{equation}
If $\I$ is prime and $\omeg \in V(\I)$, then $|\I(\omeg)|=0$ \cite[Chap.~3, Cor.~4.10]{NP01}.

Finally, the quantities $|\I(\omeg)|$ and ${\rm dist}(\omeg,\I)$ are related by the following inequality \cite[Lem.~4.1.4]{Ni_Liv}.

\begin{lem}
\label{lem:Livouille_Elimination}
Let $\I \subset \Z[\X]$ be a proper unmixed homogeneous ideal and set $r=m+1-\hgt(\I)$. For every $\omeg \in \C^{m+1}\setminus\{0\}$, we have
\[
\deg(\I)\,\log {\rm dist}(\omeg,\I)
\leq
\frac{1}{r}\log |\I(\omeg)| + 3m^3\deg(\I)\,.
\]
\end{lem}

The proof of Theorem~\ref{thm:measure_system} relies on an \textit{arithmetic Bézout theorem} \cite[Lem.~4.1.3]{Ni_Liv}. It ensures that the intersection of a projective variety with a hypersurface, both sufficiently close to a point $\omeg$, remains close to $\omeg$. For simplicity, we state it under the assumption $|\omeg|=1$.

\begin{prop}
\label{prop:Bezout}
Let $\omeg \in \C^{m+1}$ with $|\omeg| = 1$. 
Let $\p \subset \Z[\X]$ be a proper prime homogeneous ideal and let $X>0$ satisfy
\[
 \log |\mathfrak p(\omeg)| \leq -X\,.
\]
Let $Q \in \Z[\X]$ be a non-zero homogeneous polynomial. Assume that
\begin{eqnarray}\label{eq:cond_appl_Bezout1}
\nonumber \log H(Q) + (2m+2)\log(\deg(Q)+1) 
\leq 
-\log |Q(\omeg)|\hspace{3cm} \\
\hspace{0.5cm}\leq
\min\Big\{X , -\tfrac{1}{2} \log {\rm dist}(\omeg,\mathfrak p)\Big\}\, .
\end{eqnarray}
Set $r:=m + 1-\hgt(\p)$ and
\[
\sigma
=
\frac{\min \{X , -\tfrac{1}{2} \log {\rm dist}(\omeg,\mathfrak p)\}}{-\log |Q(\omeg)|} \,\cdot
\]
Then the following holds. 
\begin{itemize}
\item[{\rm (i)}] If $r=1$, then
\[
\frac{X}{2\sigma}
\leq
\log H(\mathfrak p)\deg(Q)+ \log H(Q)\deg(\mathfrak p) + 8m^2\deg(\mathfrak p)\deg(Q) \,.
\]

\item[{\rm (ii)}]  If $r\geq 2$, there exists a proper unmixed homogeneous ideal $\I$ such that $\hgt(\I)=\hgt(\mathfrak p)+1$ and
\begin{align*}
\deg(\I)&\leq \deg(\mathfrak p)\deg(Q)\,,\\
\log H(\I)
&\leq
\log H(\mathfrak p)\deg(Q)
+ \log H(Q)\deg(\mathfrak p)
+ (r+1)m\deg(\mathfrak p)\deg(Q)\,,\\
\log |\I(\omeg)|
&\leq 
-\frac{X}{2\sigma}+
\log H(\mathfrak p)\deg(Q)
+ \log H(Q)\deg(\mathfrak p) \\
&\hspace{0.5cm}+
8m^2\deg(\mathfrak p)\deg(Q).
\end{align*}
\end{itemize}
\end{prop}

The final tool we need is the following proposition, which controls the degrees, heights, and absolute values of the associated proper prime ideals of a proper unmixed homogeneous ideal. It is a slightly weaker version of \cite[Lem.~4.1.2]{Ni_Liv}, obtained by omitting the additional term $\log\vert b\vert$, which plays no role in our argument.

\begin{prop}\label{prop:decompo_primaire}
Let $\I\subset \Z[\X]$ be a proper unmixed homogeneous ideal, and let
$\p_1,\ldots,\p_s$ be its associated proper prime ideals. For
$1\le i\le s$, let $k_i$ denote the exponent of $\p_i$ in the primary
decomposition of $\I$. Then
\begin{align*}
\sum_{i=1}^s k_i\deg(\p_i)
&=
\deg(\I)\,,
\\
\sum_{i=1}^s k_i\log H(\p_i)
&\le
\log H(\I)+m^2\deg(\I)\,.
\end{align*}
Moreover, for any $\omeg\in\C^{m+1}\setminus\{0\}$, we have
\[
\sum_{i=1}^s k_i\log |\p_i(\omeg)|
\le
\log |\I(\omeg)|+m^3\deg(\I)\,.
\]
\end{prop}

\subsection{A weaker version of Theorem~\ref{thm:measure_system}}\label{sec:first_regime}

To clarify the exposition, we first state and prove a weaker version of Theorem~\ref{thm:measure_system}. 
We explain in Section~\ref{sec:2nd_regime} how to recover the sharper bound stated there.

\begin{thm}
\label{thm:measure_1st_regime}
Theorem~\ref{thm:measure_sans_system} holds with
\[
C_2(d):=e^{-C_1 d^{3t+3}}\,.
\]
\end{thm}

Recall that $t$ denotes the transcendence degree of the field extension of $\Q(z)$ generated by $f_1(z),\ldots,f_m(z)$. We also recall that the coefficients of an $M$-function, say
\[
f(z)=\sum_{n=0}^\infty u(n)z^n\,,
\]
generate a number field $\K$, that is, a finite extension of $\Q$. Moreover, there exist real numbers $\kappa,\rho>0$ and an integer $D>0$ such that
\begin{equation}\label{eq:estimation_coeff}
|\sigma(u_n)| \leq \kappa \rho^{n}\,,
\qquad
D^n u_n \in \Z_{\K}\,,
\qquad
\forall n\geq 1,\ \forall \sigma \in {\rm Gal}(\K/\mathbb Q)\, ,
\end{equation}
where $\Z_\K$ denotes the ring of integers of $\K$. This follows, for instance, from \cite[Thm.~1.1]{ABS} and \cite[Chap.~3, Cor.~8]{Du93}; see also \cite[p.~145]{Ni_Liv}. 

The proof relies on the following lemma, which will be established in the next section.

\begin{lem}\label{lem:Ideal_Mahler} 
We keep the notation and assumptions of Theorem~\ref{thm:measure_system}, and set
\[
\omeg_{0}:=(1,f_1(\alpha),\ldots,f_m(\alpha)) \in \C^{m+1}.
\]
There exist positive constants $\lambda$ and $D_0$ such that for every $D>D_0$, every $H$ satisfying $\log H \geq D^{2t+3}$, and every proper unmixed homogeneous ideal $\I \subset \Z[\X]$ with
\[
 \hgt(\I)\geq m+1-t\,,\qquad
\deg(\I) \leq D\,,\qquad \mbox{and} \qquad
\log H(\I) \leq \log H\,,
\]
we have
\[
\log |\I(\omeg_0)| \geq -\lambda D^{t}\log H\,.
\]
\end{lem}

\begin{rem}
Let $\p$ be the ideal generated by the homogeneous polynomials vanishing at $\omeg_0$.  Then $\p$ is proper, and it is also unmixed since it is prime. 
Moreover, 
$\log |\p(\omeg_0)|=-\infty$.
Consequently, no lower bound of the form given in Lemma~\ref{lem:Ideal_Mahler} can hold for $\p$. It follows that
$
\hgt(\p)<m+1-t$, 
or equivalently
$
\hgt(\p)\le m-t$. 
Since, on the other hand, one always has $\hgt(\p)\ge m-t$, we obtain
\[
\hgt(\p)=m-t\,.
\]
This is equivalent to Nishioka's theorem~\cite[Thm.~4.2.1]{Ni_Liv}: 
\[
{\rm tr.deg}_{\Q}(f_1(\alpha),\ldots,f_m(\alpha))=t=:{\rm tr.deg}_{\Q(z)}(f_1(z),\ldots,f_m(z))\,.
\]
 Thus, Lemma~\ref{lem:Ideal_Mahler} contains Nishioka's theorem as a special case. It also explains why the assumption
\[
\hgt(\I)\ge m+1-t
\]
is necessary.
\end{rem}

\begin{proof}[Proof of Theorem~\ref{thm:measure_1st_regime}]
Let $\omeg_0=(1,f_1(\alpha),\ldots,f_m(\alpha))$ as in Lemma~\ref{lem:Ideal_Mahler}. 
There is no restriction in assuming that $|\omeg_0| = 1$. 
We can also assume that $t\geq 1$, since otherwise the result follows from the classical Liouville inequality. 
Throughout the proof, $\lambda_1,\lambda_2,\ldots$ denote positive constants independent of $P$. 

Let $\prescript{\rm h}{}{P}\in \Z[\X]$ be the homogenization of $P$, so that
\[
\prescript{\rm h}{}{P}(\omeg_0)=P(f_1(\alpha),\ldots,f_m(\alpha))\,,
\]
and 
\[
\deg(\prescript{\rm h}{}{P})=\deg(P)=d\,,\quad
H(\prescript{\rm h}{}{P})=H(P)\,.
\]
Assume that $|\prescript{\rm h}{}{P}(\omeg_0)|\neq 0$. We aim to bound this quantity from below. 
We also assume that
\[
-\log |\prescript{\rm h}{}{P}(\omeg_0)| \geq \log H(P) + (2m+2)\log(d+1)\,,
\]
since otherwise there is nothing to prove. 

As we will see, if $t=m$, then the result follows from Lemma~\ref{lem:Ideal_Mahler} applied with $\I=(\prescript{\rm h}{}{P})$. But when $t<m$, the height of 
$(\prescript{\rm h}{}{P})$ is not large enough to apply Lemma~\ref{lem:Ideal_Mahler} and then we need to use the arithmetic Bezout theorem to 
provide an ideal with a larger height and apply the lemma.

Let $\p$ denote the homogeneous ideal associated with $\omeg_0$, that is,
\[
\p=\{R\in \Z[\X]\ \text{homogeneous} : R(\omeg_0)=0\}\,.
\]
Then $\p$ is proper and prime. 
By Nishioka's theorem \cite[Thm.~4.2.1]{Ni_Liv}, we have  
$
\hgt(\p)= m-t$. 
In particular, \[m+1-\hgt(\p)=1+t\ge 2\,.\]
Moreover, $\deg(\p)$ and $H(\p)$ depend only on $\omeg_0$, and since $\omeg_0\in V(\p)$ we have
\[
{\rm dist}(\omeg_0,\p)=0 \qquad \mbox{and}\qquad|\p(\omeg_0)|=0
\]
(see \cite[Cor.~4.10]{NP01}).

Assume first that $\p \neq (0)$. We apply Proposition~\ref{prop:Bezout} with $\p$ and $Q=\prescript{\rm h}{}{P}$.
We set
\[
X = -\log |\prescript{\rm h}{}{P}(\omeg_0)| < +\infty\,.
\]
Then $\sigma=1$, and the assumptions of Proposition~\ref{prop:Bezout} are satisfied.
Since $m+1-\hgt(\p)\ge 2$, we obtain a proper unmixed homogeneous ideal $\I$ such that $\hgt(\I)=\hgt(\p)+1$ and
\begin{align}
\deg(\I)&\leq \lambda_1 \deg(P), \nonumber\\
\log H(\I)&\leq \lambda_1 \deg(P) + \lambda_2 \log H(P), \label{eq:mino_intersec_ideals}\\
\log |\I(\omeg_0)| &\leq \tfrac12 \log |\prescript{\rm h}{}{P}(\omeg_0)| + \lambda_1 \deg(P) + \lambda_2 \log H(P). \nonumber
\end{align}

If $\p=(0)$, we set $\I=(\prescript{\rm h}{}{P})$, which satisfies the same bounds by
\eqref{eq:deg_H_principal} and \eqref{eq:absolute_value_principale}.

Let
\[
D=\lambda_1 \deg(P) \qquad \mbox{and} \qquad
\log H=\lambda_1 \deg(P) + \lambda_2 \log H(P)\,.
\]
Up to enlarging $\lambda_1$, we assume that $\lambda_1\ge D_0$, so that $D\ge D_0$, where $D_0$ is given by Lemma~\ref{lem:Ideal_Mahler}. Note that, in both cases, we have $m+1-\hgt(\I)\le t$.

We first consider the special case where 
\[
\lambda_1 \deg(P) + \lambda_2 \log H(P) \ge (\lambda_1 \deg(P))^{2t+3}\,.
\]
Then $\log H \ge \lambda_3 D^{2t+3}$ for some $\lambda_3>0$ which does not depend on $D$. 
Multiplying $P$ by a $\lceil \lambda_3^{-1} \rceil$, which does not affect the conclusion of the theorem, we can assume that $\log H \geq D^{2t+3}$. 
Applying Lemma~\ref{lem:Ideal_Mahler} gives
\[
\log \vert\I(\omeg_0) \vert \geq  -\lambda D^{t}\log H \geq -\lambda_4 \left( \deg(P)^t\log H(P) +  \deg(P)^{t+1}\right)\,,
\]
Using \eqref{eq:mino_intersec_ideals}, we obtain
\[
\log \vert P(f_1(\alpha),\ldots,f_m(\alpha))\vert \geq   -\lambda_5 \left( \deg(P)^t\log H(P) +  \deg(P)^{t+1}\right)\,.
\]
In that case, the conclusion of the theorem holds with $C_1=\lambda_5$ and the sharper bound $C_2(d)=e^{-C_1d^{t+1}}$. 

To prove the general case, we proceed as follows. 
Let $\lambda_6>0$ be large enough and set
\[
P_0 = \left\lceil e^{\lambda_6 \deg(P)^{2t+3}} \right\rceil P\,.
\]
Then $\log H(P_0)\le \log H(P)+\lambda_6 \deg(P)^{2t+3}$, and applying the previous bound to $P_0$ yields
	\begin{align*}
	\log |P(f_1(\alpha),\ldots,f_m(\alpha))|
&\ge \log |P_0(f_1(\alpha),\ldots,f_m(\alpha))| - \lambda_6 \deg(P)^{2t+3}\\
&\ge -\lambda_5\big(\deg(P)^t \log H(P_0) + \deg(P)^{t+1}\big) \\
&\hspace{5cm}- \lambda_6 \deg(P)^{2t+3}\\
&\ge -\lambda_7\big(\deg(P)^t \log H(P) + \deg(P)^{3t+3}\big)\,.
	\end{align*}
	
\noindent This completes the proof.
\end{proof}

\subsection{Proof of Lemma~\ref{lem:Ideal_Mahler}}
 
We continue with the assumptions of Theorem~\ref{thm:measure_system}. Recall that we set
\[
\omeg_0=(1,f_1(\alpha),\ldots,f_m(\alpha))
\quad\text{and}\quad
t={\rm tr.deg}_{\Q(z)}(f_1(z),\ldots,f_m(z))\,,
\]
while, according to Nishioka's theorem, $t= {\rm tr.deg}_{\Q}(\omeg_0)$.
For $r \in \{1,\ldots,t\}$, we introduce the quantities
\[
\nu(r)=\frac{r}{t-r+1}\,,\qquad
\mu(r)=\frac{3t-r+3}{t-r+1}\,,\qquad
\tau(r)=\frac{2t^2+2t-2rt-2}{r(t-r+1)}\,\cdot
\]
When $r=t$, we have $\nu(t)=t$, $\mu(t)=2t+3$, and $\tau(t)=0$. In particular, Lemma~\ref{lem:Ideal_Mahler} follows from the more precise statement below.

\begin{lem}\label{lem:Ideal_Mahler2}
Let $r \in \{1,\ldots,t\}$. There exist positive constants $\lambda$ and $D_0$ such that for every $D\geq D_0$ and every $H$ satisfying
\[
\log H \geq \lambda^{\tau(r)} D^{\mu(r)},
\]
the following holds. For every proper unmixed homogeneous ideal $\I \subset \Z[\X]$ with
\[
\hgt(\I)=m+1-r\,,\qquad
\deg(\I) \leq D\,,\qquad
H(\I) \leq H\,,
\]
we have
\begin{equation}\label{eq:mino_Iomega}
\log |\I(\omeg_0)| \geq -\lambda \Bigl(D^{\nu(r)}\log H(\I)
+ D^{\nu(r)-1}\log H \cdot \deg(\I)\Bigr)\,.
\end{equation}
\end{lem}

This lemma is the analogue of \cite[Prop.~4.4.9]{Ni_Liv}, with the parameters $t,\lambda,D,\log H$ playing the roles of $m$, $\lambda^r$, $\lambda^{m-r}D^{m-r+1}$, and $\lambda^{m-r}D^{m-r}\log H$, respectively.

The strategy of the proof is as follows. We first construct a dense sequence of polynomials $Q_{N,k}\in \mathbb Z[\X]$ taking small values at $\omeg_0$. We then argue by contradiction: assuming that the conclusion fails for some ideal $\I$ with maximal height, we apply the arithmetic Bézout theorem to a suitable prime ideal associated with $\I$ and to one of the polynomials $Q_{N,k}$. This produces an ideal of larger height for which the lemma also fails, contradicting the maximality of the height. 

Throughout the proof, we let $\gamma_1,\gamma_2,\gamma_3,\ldots,\kappa_1,\kappa_2,\ldots,\eta_1,\eta_2,\ldots$ denote positive constants depending only on $f_1,\ldots,f_m$ and $\alpha$. Without loss of generality, we assume that $|\omeg_0|=1$.

\subsubsection{A multiplicity lemma}

In order to control the value of a polynomial at $\omeg_0$, we will use a refined version of Nishioka's multiplicity lemma \cite[Thm.~4.3]{Ni_Liv}. We let $\val$ denote the usual valuation, associated with the ideal $(z)$, for elements of $\C[[z]]$. The following result is proved in \cite[Thm.~V.1.1]{Fe19}.

\begin{lem}[Multiplicity Lemma]\label{lem:Multiplicity}
Let $M,N\geq 1$ be integers, and let 
$R_0 \in \C[z,\X]$ 
be such that $\deg_z (R_0)\leq M$ and $\deg_{\X}(R_0)\leq N$. If
\[
R(z):=R_0(z,f_1,\ldots,f_m)\neq 0\,,
\]
then
\[
\val (R) = \mathcal O(MN^{t})\,,
\]
where the implied constant depends only on $f_1,\ldots,f_m$.
\end{lem}

\begin{rem}
Alternatively, this result can be obtained by a slight modification of the proof of \cite[Thm.~4.3]{Ni_Liv}. 
Let $\mathfrak p_z\subset \C[z][\X]$ denote the homogeneous ideal generated by the homogeneous polynomials vanishing at $(1,f_1,\ldots,f_m)$. 
Let $\prescript{\rm h}{}{R_0}\in \C[z][\X]$ be the homogenization of $R_0$ (with respect to $\X$). The key modification is to replace the ideal $\I$ used in \cite{Ni_Liv} by the ideal obtained from the arithmetic Bézout theorem over $\C(z)$ \cite[Thm.~4.1.7]{Ni_Liv}, 
applied to $\mathfrak p_z$ and $\prescript{\rm h}{}{R_0}$.
\end{rem}
 
 \subsubsection{A dense sequence of good polynomial approximations}

The purpose of Proposition~\ref{prop:polynôme_aux} is to construct
many polynomials taking small values at $\omeg_0$, with respect to their degrees and heights.

\begin{prop}
\label{prop:polynôme_aux}
There exist positive integers $\gamma_1,\ldots,\gamma_{6}$ such that, for every $N\geq \gamma_1$ and every integer $k$ with $q^k\geq \gamma_2 N^{t+1}$, there exists a homogeneous polynomial $Q_{N,k}\in \Z[\X]$ satisfying \eqref{eq:cond_appl_Bezout1} and
\begin{equation}\label{eq:deg_height_Qk}
\deg(Q_{N,k}) \leq \gamma_{3}N,\qquad \log H(Q_{N,k})\leq \gamma_{4}Nq^k\,,
\end{equation}
and
\begin{equation}\label{eq:encadr_Qk}
-\gamma_{5}N^{t+1}q^k \leq \log | Q_{N,k}(\omeg_0) | \leq -\gamma_{6}N^{t+1}q^k\,.
\end{equation}
\end{prop}

\begin{proof}

This proposition corresponds to \cite[Prop.~4.4.7]{Ni_Liv}. The main difference is that we do not assume that the functions $f_1,\ldots,f_m$ are algebraically independent. The parameter $m$, which  corresponds to the transcendence degree of these functions in \cite{Ni_Liv}, must therefore be replaced by $t$ in \eqref{eq:encadr_Qk}. Accordingly, we may assume without loss of generality that $f_1,\ldots,f_t$ are algebraically independent over $\Q(z)$.

Let $\K$ be a number field containing $\alpha$, the coefficients of $f_1,\ldots,f_m$, and those of the matrix $A(z)$. Let $\OK$ denote its ring of integers. We proceed as in \cite[p.~139--141]{Ni_Liv}, but construct a Padé approximant for $1,f_1,\ldots,f_t$, denoted by
\[
R(z)=R_0(z,1,f_1(z),\ldots,f_t(z))\,,
\]
where $R_0(z,\X)\in \OK[z,\X]$ is homogeneous in $\X$, and satisfies
\begin{eqnarray}\label{eq:R0}
\nonumber \deg_z(R_0)\leq N\,,\qquad \deg_{\X}(R_0)\leq N\,,
\\
\log \overline{H}(R_0)=\mathcal O(N\log N)\,,  \qquad 
\val(R)\geq \kappa_1 N^{t+1},
\end{eqnarray}
for some constant $\kappa_1>0$. Here $\overline{H}(R_0)$ denotes the maximum of the absolute values of the coefficients of $R_0$ and their Galois conjugates. 

The existence of such an $R_0$ follows from Siegel's lemma together with the estimates \eqref{eq:estimation_coeff}. Since Lemma~\ref{lem:Multiplicity} (applied with $M=N$) gives $\val(R)=\mathcal O(N^{t+1})$, we deduce that
\begin{equation}\label{eq:auxiliary_func}
-\kappa_2 q^k N^{t+1} \leq \log |R(\alpha^{q^k})| \leq -\kappa_3 q^k N^{t+1}\,,
\end{equation}
for all $N\geq \kappa_4$ and $q^k\geq \kappa_5 N^{t+1}$, for some $\kappa_2,\ldots,\kappa_5>0$.

Let $a(z)$ be such that $a(z)A(z)$ has entries in $\OK[z]$. Using the Mahler system
\[
\begin{pmatrix}
1\\ f_1(z^q)\\ \vdots \\ f_m(z^q)
\end{pmatrix}
=
B(z)
\begin{pmatrix}
1\\ f_1(z)\\ \vdots \\ f_m(z)
\end{pmatrix}\,,
\qquad B(z):=\left(\begin{array}{c|ccc} 1 & 0 &\cdots &0\\ \hline 0 &&& \\ \vdots &&A(z) &
	\\ 0 & &&\end{array}\right)\,,
\]
we define recursively a sequence $(R_k(z,\X))_{k\geq 0}\subset \OK[z,\X]$ by
\begin{equation}\label{eq:R_krec}
R_{k+1}(z,\X):=a(z)^N R_k(z^q,B(z)\X)\,.
\end{equation}
Then
\begin{equation}\label{eq:R_k}
R_k(z,1,f_1(z),\ldots,f_t(z)) = a_k(z)^N R(z^{q^k})\,,\qquad k\geq 0\,,
\end{equation}
where $a_k(z)=a(z)a(z^q)\cdots a(z^{q^{k-1}})$ and $a_0(z)=1$.

We now construct polynomials with integer coefficients. By construction, $\deg_z R_k \leq \kappa_6 q^k N$, for some $\kappa_6>0$. Let $\xi$ be an algebraic integer generating $\K$ over $\Q$, and let $a\in \Z$ be such that $a\alpha \in \Z[\xi]$ and $a\OK \subset \Z[\xi]$. For each $k\geq 0$, let $P_k \in \Z[Y,\X]$ be the polynomial, homogeneous in $\X$, such that
\[
P_k(\xi,\X)=a^{\lceil \kappa_6 q^k N\rceil} R_k(\alpha,\X)\,,
\qquad
\deg_Y(P_k) \leq [\K:\Q]\,.
\]

Combining \eqref{eq:R0}, \eqref{eq:auxiliary_func}, \eqref{eq:R_krec}, and \eqref{eq:R_k}, we obtain
\[
\deg_{\X} P_k \leq N,\qquad \log H(P_k)=\mathcal O(q^k N)\,,
\]
and
\[
-\kappa_7 q^k N^{t+1} \leq \log |P_k(\xi,\omeg_0)| \leq -\kappa_8 q^k N^{t+1}\,.
\]

We may then apply \cite[Lem.~4.1.9]{Ni_Liv} to the polynomials $P_k$ to obtain the desired polynomials $Q_{N,k}$.
\end{proof}

\subsubsection{Proof of Lemma~\ref{lem:Ideal_Mahler2}}

We argue by contradiction. Assume that the lemma does not hold, and let $r_1$ be the smallest integer $r\in \{1,\ldots,t\}$ for which it fails. Then there exist arbitrarily large real numbers $\lambda_1$ such that, for each $\lambda_1$, there exist arbitrarily large real numbers $D_1$, a real number $H_1$ satisfying
\begin{equation}\label{eq:H0}
\log H_1 \geq \lambda_1^{\tau(r_1)}D_1^{\mu(r_1)}\,,
\end{equation}
and a proper unmixed homogeneous ideal $\I$ such that
\begin{align}\label{eq:majo_value_I}
\hgt(\I)&=m+1-r_1,\quad \deg(\I)\leq D_1,\quad H(\I)\leq H_1\,,\\
\nonumber
\log |\I(\omeg_0)| &< -\lambda_1 \left(D_1^{\nu(r_1)}\log H(\I)+D_1^{\nu(r_1)-1}\log H_1\,\deg(\I)\right)\,.
\end{align}

Fix such $\lambda_1,D_1,H_1$, and $\I$. Using the primary decomposition of $\I$ and Proposition~\ref{prop:decompo_primaire}, we first extract a prime ideal with similar properties.

Let $\p_1,\ldots,\p_s$ be the proper prime ideals associated with $\I$. We claim that there exists $i_1\in\{1,\ldots,s\}$ such that
\[
\log |\p_{i_1}(\omeg_0)| < -\tfrac{1}{2}\lambda_1 \left(D_1^{\nu(r_1)}\log H(\p_{i_1}) + D_1^{\nu(r_1)-1}\log H_1\,\deg(\p_{i_1})\right).
\]

Suppose the contrary. Then, by Proposition~\ref{prop:decompo_primaire},
\begin{align*}
\log |\I(\omeg_0)|
&\geq \sum_{i=1}^s k_i \log |\p_i(\omeg_0)| - m^3\deg(\I)\\
&> -\tfrac{1}{2}\lambda_1 \Bigl(D_1^{\nu(r_1)} \sum_{i=1}^s k_i \log H(\p_i)
+ D_1^{\nu(r_1)-1}\log H_1 \sum_{i=1}^s k_i \deg(\p_i)\Bigr) \\
&\hspace{8cm} - m^3\deg(\I)\\
&\geq -\tfrac{1}{2}\lambda_1 \Bigl(D_1^{\nu(r_1)}\log H(\I)
+ D_1^{\nu(r_1)}m^2\deg(\I)
\\ &\hspace{3.8cm}+ D_1^{\nu(r_1)-1}\log H_1\,\deg(\I)\Bigr) 
  - m^3\deg(\I)\,.
\end{align*}

Combining with \eqref{eq:majo_value_I}, we obtain
\begin{eqnarray}\label{eq:D_1}
\tfrac{1}{2}\lambda_1\left(D_1^{\nu(r_1)}\log H(\I)
+ D_1^{\nu(r_1)-1}\log H_1\,\deg(\I)\right) \hspace{2.5cm}  \\
\nonumber \hspace{6.4cm}< \tfrac{1}{2}\lambda_1 D_1^{\nu(r_1)}m^2\deg(\I) + m^3\deg(\I)\,.
\end{eqnarray}

Since $\mu(r_1)>1$, \eqref{eq:H0} implies $\log H_1 \geq D_1^{1+\varepsilon}$ for some $\varepsilon>0$. Hence \eqref{eq:D_1} is impossible for $D_1$ large enough. This proves the claim.

We now set $\p=\p_{i_1}$ and define
\begin{equation}\label{eq:mino_prime_ideal}
X := \tfrac{1}{2}\lambda_1 \left(D_1^{\nu(r_1)}\log H(\p)
+ D_1^{\nu(r_1)-1}\log H_1\,\deg(\p)\right)
> -\log |\p(\omeg_0)|.
\end{equation}

Since $\p$ is a proper prime ideal, it is unmixed. Hence
Proposition~\ref{prop:decompo_primaire} applies and yields
\begin{equation}\label{eq:parameter_p0}
\deg(\p)\le D_1\,,\qquad
\log H(\p)\le \log H_1 + m^2D_1 \le 2\log H_1\,,
\end{equation}
provided that $D_1$, and hence $H_1$, is sufficiently large. Using the prime ideal $\p$ together with a suitably chosen polynomial $Q$ from the sequence $(Q_{N,k})$ provided by Proposition~\ref{prop:polynôme_aux}, we will apply the arithmetic Bézout theorem. If $r_1 \geq 2$, this yields a proper unmixed homogeneous ideal $\J$ together with upper bounds for $\deg(\J)$, $\log H(\J)$, and $\log |\J(\omeg_0)|$. On the other hand, since $r_1$ is minimal, Lemma~\ref{lem:Ideal_Mahler2} applies to $\J$ and provides a lower bound for $\log |\J(\omeg_0)|$, which contradicts the preceding upper bound. If $r_1=1$, then the inequality corresponding to the case $r=1$ in Proposition~\ref{prop:Bezout} will be incompatible with \eqref{eq:majo_value_I}. 

The main difficulty is to choose the parameters $N$ and $k$ appropriately. On the one hand, they must be large enough for Proposition~\ref{prop:polynôme_aux} to apply. On the other hand, while $N$ must be sufficiently large, the quantity $q^kN^{t+1}$ must remain small enough for the arithmetic Bézout theorem to be applicable. Finally, these parameters must be balanced so as to produce a contradiction. This explains our choice of the functions $\nu(r)$, $\mu(r)$, and $\tau(r)$.

We now choose parameters $N$ and $k$. 
Let $\gamma_1,\ldots,\gamma_6$ be the positive constants given by Proposition~\ref{prop:polynôme_aux}. Assume that $D_1$ is large enough so that
\begin{equation}\label{eq:def_N}
N:= \left\lfloor \lambda_1^{\frac{1}{2t+2}} D_1^{\frac{1}{t-r_1+1}} \right\rfloor \geq \gamma_1\,.
\end{equation}
For $D_1$ sufficiently large, we have
\[
\gamma_2\gamma_5 q N^{2t+2}
\leq
\frac{1}{5r_1}\lambda_1^{\tau(r_1)+1}
D_1^{\frac{2t+r_1+2}{t-r_1+1}}\,.
\]
On the other hand, it follows from Lemma~\ref{lem:Livouille_Elimination},
\eqref{eq:H0}, and \eqref{eq:mino_prime_ideal} that
\[
-\frac{1}{2}\log {\rm dist}(\omeg_0,\p)
>
\frac{1}{4r_1}\lambda_1 D_1^{\nu(r_1)-1}\log H_1 - 3m^2
\geq
\frac{1}{5r_1}\lambda_1^{\tau(r_1)+1}
D_1^{\frac{2t+r_1+2}{t-r_1+1}}\,.
\]
Therefore, by the definition of $X$,
\begin{equation*}
\min \left\{X,-\frac{1}{2}\log {\rm dist}(\omeg_0,\p)\right\}
\geq
\frac{1}{5r_1}\lambda_1^{\tau(r_1)+1}
D_1^{\frac{2t+r_1+2}{t-r_1+1}}
\geq
\gamma_2\gamma_5 q N^{2t+2}\,.
\end{equation*}

Let $k$ be the unique integer such that
\begin{equation}\label{eq:choice_k}
\gamma_5 N^{t+1} q^k
\leq
\min \left\{X,-\frac{1}{2}\log {\rm dist}(\omeg_0,\p)\right\}
\leq
\gamma_5 N^{t+1} q^{k+1}\,.
\end{equation}
Then
\[
\gamma_5 N^{t+1} q^{k+1} \geq \gamma_2\gamma_5 q N^{2t+2}\,,
\]
and therefore
\begin{equation}\label{eq:majo_N}
q^k \geq \gamma_2 N^{t+1}\,.
\end{equation}
By \eqref{eq:def_N} and \eqref{eq:majo_N}, Proposition~\ref{prop:polynôme_aux}
applies and yields a polynomial $Q=Q_{N,k}$.
Moreover, Proposition~\ref{prop:polynôme_aux} together with
\eqref{eq:choice_k} shows that the condition \eqref{eq:cond_appl_Bezout1}
in Proposition~\ref{prop:Bezout} is satisfied.

From Proposition~\ref{prop:polynôme_aux}, we obtain
\begin{multline}\label{eq:majo_pour_J}
\log H(\p)\deg(Q)+\log H(Q)\deg(\p) \\ \hspace{1.7cm}\max\{8m^2;(r_1+1)m\}\deg(Q)\deg(\p) 
 \leq \eta_1N(\log H(\p)+q^k\deg(\p))\,,
\end{multline}
for some $\eta_1>0$.

Let
\[
\sigma := \frac{\min\left\{X,-\tfrac{1}{2}\log {\rm dist}(\omeg_0,\p)\right\}}{-\log |Q(\omeg_0)|}\,\cdot
\]
Proposition~\ref{prop:polynôme_aux} and \eqref{eq:choice_k} imply that
\[
\sigma \le \frac{\gamma_5}{\gamma_6}\,q.
\]
In particular, $\sigma$ is bounded independently of $\lambda_1$ and $D_1$.

We claim that, if $\lambda_1$ is sufficiently large, then
\begin{equation}\label{eq:claim}
\eta_1 N \log H(\p) \le \frac{X}{8\sigma}
\qquad \mbox{and} \qquad
\eta_1 N q^k \deg(\p) \le \frac{X}{8\sigma}\,\cdot
\end{equation}
We now prove this claim.

By the definition of $X$, and since $\nu(r_1)$ is greater than or equal to the exponent of $D_1$ in \eqref{eq:def_N}, we have
\[
\frac{X}{8\sigma}
\ge
\frac{1}{16\sigma}\lambda_1 D_1^{\nu(r_1)}\log H(\p)
\ge
\eta_1 N \log H(\p)\,,
\]
provided that $\lambda_1$ is sufficiently large. This proves the first inequality in \eqref{eq:claim}.

For the second one, observe that the left-hand inequality in \eqref{eq:choice_k}, together with $\deg(\p)\le D_1$, implies that
\[
\eta_1 N q^k \deg(\p)
\le
\eta_1 \gamma_5^{-1} X\, D_1 N^{-t}\,.
\]
By the definition of $N$, if $\lambda_1$ is sufficiently large, then
\[
\eta_1 \gamma_5^{-1} D_1 N^{-t}\le \frac{1}{8\sigma}\,,
\]
which proves the second inequality in \eqref{eq:claim}. 
 Hence,
\begin{equation}\label{eq:comparaison_X}
\eta_1 N\bigl(\log H(\p) + q^k\deg(\p)\bigr)
\leq \tfrac{X}{4\sigma}\,\cdot
\end{equation}

We now distinguish two cases.

Suppose first that $r_1=1$. Then Proposition~\ref{prop:Bezout}, together with \eqref{eq:majo_pour_J}, yields
\[
\frac{X}{2\sigma} \leq \eta_1 N\bigl(\log H(\p)+q^k\deg(\p)\bigr)\, ,
\]
which contradicts \eqref{eq:comparaison_X}.

Suppose now that $r_1\geq 2$. Then Proposition~\ref{prop:Bezout}, together with \eqref{eq:majo_pour_J}, yields a proper unmixed homogeneous ideal $\J$ with $\hgt(\J)=\hgt(\p)+1$ such that
\begin{equation}\label{eq:deg_H_J}
\deg(\J) \leq \deg(\p)\deg(Q),\qquad
\log H(\J)\leq \eta_1N\bigl(\log H(\p)+q^k\deg(\p)\bigr)\,,
\end{equation}
and
\begin{equation}\label{eq:mino_Jomeg}
-\log |\J(\omeg_0)|
\geq
\frac{X}{2\sigma}-\eta_1N\bigl(\log H(\p)+q^k\deg(\p)\bigr)
\geq
\frac{X}{4\sigma}\,,
\end{equation}
where the last inequality follows from \eqref{eq:comparaison_X}.

Using Proposition~\ref{prop:polynôme_aux}, \eqref{eq:parameter_p0}, and \eqref{eq:deg_H_J}, we obtain
\begin{equation}\label{eq:deg_J}
\deg(\J)\leq \gamma_3N\deg(\p)\leq \gamma_3ND_1=:D\,.
\end{equation}
Furthermore, \eqref{eq:mino_prime_ideal}, \eqref{eq:def_N}, and \eqref{eq:choice_k} imply that
\[
q^kD_1\leq \eta_2 \lambda_1^{1/2}\log H_1
\]
for some $\eta_2>0$. Hence, combining \eqref{eq:majo_pour_J} and \eqref{eq:deg_H_J}, we obtain
\begin{equation}\label{eq:H_J}
\log H(\J)\leq \eta_3 \lambda_1^{1/2}N\log H_1 =: \log H
\end{equation}
for some $\eta_3>0$.

Set $r=r_1-1 \in \{1,\ldots,t-1\}$. Since $r<r_1$, Lemma~\ref{lem:Ideal_Mahler2} holds for $r$. Let $\lambda$ and $D_0$ be the corresponding constants. Assume that $D_1$ is large enough so that $D \ge D_0$.

Using \eqref{eq:H0}, \eqref{eq:def_N}, and \eqref{eq:H_J}, one checks that
\[
\log H \ge \lambda^{\tau(r)}D^{\mu(r)}\,,
\]
provided that $\lambda_1$ is sufficiently large.
Since $\hgt(\J)=m+1-r$, the ideal $\J$ satisfies the assumptions of Lemma~\ref{lem:Ideal_Mahler2}, and therefore
\[
-\log |\J(\omeg_0)|
\le
\lambda\left(D^{\nu(r)}\log H(\J)+D^{\nu(r)-1}\log H\,\deg(\J)\right)\,.
\]

Using moreover the bound
\[
q^kD_1\le \eta_2 \lambda_1^{1/2}\log H_1\,,
\]
the upper bounds \eqref{eq:deg_H_J}, the estimate $\deg(Q)\le \gamma_3N$, and the definitions of $D$ and $\log H$, we obtain the existence of positive constants $\eta_4,\eta_5$ such that
\begin{align*}
-\log |\J(\omeg_0)|
&\le
\eta_4 N^{\nu(r)+1}D_1^{\nu(r)}
\bigl(\log H(\p)+q^k\deg(\p)+D_1^{-1}\log H_1\,\deg(\p)\bigr)
\\
&\le
\eta_5 \lambda_1^{1/2}N^{\nu(r)+1}D_1^{\nu(r)}
\bigl(\log H(\p)+D_1^{-1}\log H_1\,\deg(\p)\bigr)\,.
\end{align*}

Since
\[
N^{\nu(r)+1}D_1^{\nu(r)}
\le
\lambda_1^{\frac{\nu(r)+1}{2t+2}}D_1^{\nu(r_1)}
\qquad\text{and}\qquad
\frac{\nu(r)+1}{2t+2}<\frac12 \,,
\]
it follows that
\begin{align*}
-\log |\J(\omeg_0)|
&\le
\eta_5\lambda_1^{1-\varepsilon}
\left(D_1^{\nu(r_1)}\log H(\p)+D_1^{\nu(r_1)-1}\log H_1\,\deg(\p)\right)
\\
&\le
2\eta_5\lambda_1^{-\varepsilon}X
\end{align*}
for some $\varepsilon>0$. This contradicts \eqref{eq:mino_Jomeg} when $\lambda_1$ is sufficiently large.
\qed

\subsection{End of the proof of Theorem~\ref{thm:measure_system}}\label{sec:2nd_regime}

In this section, we continue under the assumptions of Theorem~\ref{thm:measure_system} and with the notation of Section~\ref{sec:first_regime}. In particular, we set
\[
\omeg_0 := (1,f_1(\alpha),\ldots,f_m(\alpha))\,,
\]
and we let $t$ denote the transcendence degree of $\Q(\omeg_0)$.

In the proof of Theorem~\ref{thm:measure_1st_regime}, we established that one may take
\[
C_2(d)=e^{-C_1 d^{t+1}}
\]
in the conclusion of Theorem~\ref{thm:measure_system}, provided that $\log H(P)$ is sufficiently large compared with $\deg(P)$. We now explain how to obtain the bound
\[
C_2(d)=e^{-C_1 d^{2t+2}}
\]
announced in Theorem~\ref{thm:measure_system} for an arbitrary polynomial $P$. Thus, it remains to consider the case where $\log H(P)$ is small compared with $\deg(P)$. For this, we need the following variant of Lemma~\ref{lem:Ideal_Mahler2}.

\begin{lem}\label{prop:majo_ideal_bis}
Let $r \in \{1,\ldots,t\}$. There exist positive constants $\lambda$ and $D_0$ such that, for every $D\geq D_0$ and every $\delta$ satisfying
\[
\frac{2t-r+2}{t-r+1}\leq \delta \leq \frac{3t-r+3}{t-r+1}\,,
\]
the following holds. For every proper unmixed homogeneous ideal $\I \subset \Z[\X]$ with
\[
\hgt(\I)=m+1-r\,,\qquad
\deg(\I) \leq D\,,\qquad
\log H(\I) \leq D^{\delta},
\]
we have
\begin{equation}\label{eq:mino_Iomega_bis}
\log |\I(\omeg_0)| \geq -\lambda \left(D^{\frac{r}{t-r+1}}\log H(\I)+D^{\frac{r}{t-r+1}-1+\delta}\deg(\I)\right)\,.
\end{equation}
\end{lem}

\begin{proof}
The proof is analogous to that of \cite[Prop.~4.4.10]{Ni_Liv}, with the parameters $t,\lambda,D,\log H$ replacing, respectively, 
\[m, \lambda^r, \lambda^{m-r}D^{m-r+1}, \lambda^{m-r+\delta+1}D^{m-r}\log H\,,\] and with Proposition~\ref{prop:polynôme_aux2} below replacing \cite[Prop.~4.4.8]{Ni_Liv}. Since the argument is very close to that of Lemma~\ref{lem:Ideal_Mahler2}, we omit the details.
\end{proof}

\begin{prop}\label{prop:polynôme_aux2}
There exist positive real numbers $\gamma_7,\ldots,\gamma_{11}$, with $\gamma_{10}\ge 1$, such that for every $\delta \ge t+2$ and every $s \ge \gamma_7$, there exists a homogeneous polynomial
$
Q \in \Z[\X]$
satisfying \eqref{eq:cond_appl_Bezout1} and
\[
\deg(Q) \le \gamma_8 s\,,\quad
\log H(Q)\le \gamma_9 s^\delta\,,
\quad
-\gamma_{10}s^{t+\delta} \le \log |Q(\omeg_0)| \le -\gamma_{11}s^{t+\delta}\,.
\]
\end{prop}

\begin{proof}
Fix $\delta\ge t+2$ and $s\ge 1$. Let $k$ be the least integer such that
\[
q^k\ge 2\gamma_2 s^{\delta-1}\,,
\]
and set
\[
N:=\lceil s\rceil\,.
\]
If $\gamma_7$ is sufficiently large and $s \ge \gamma_7$, then $N$ and $k$ satisfy the assumptions of Proposition~\ref{prop:polynôme_aux}. The conclusion therefore follows from Proposition~\ref{prop:polynôme_aux} by taking
\[
Q=Q_{N,k}\,,
\]
and observing that
\[
Nq^k=\mathcal O(s^\delta)\,.
\]
There is no loss of generality in assuming that $\gamma_{10}\ge 1$.
\end{proof}
We are now ready to conclude the proof of Theorem~\ref{thm:measure_system}.

\begin{proof}[Proof of Theorem~\ref{thm:measure_system}]
We continue with the notation introduced in the proof of Theorem~\ref{thm:measure_1st_regime}. Let $\I$ be the ideal constructed there. Recall that, by \eqref{eq:mino_intersec_ideals}, we have
\[
\deg(\I)\le \lambda_1\deg(P)=:D\,,
\qquad
\log H(\I)\le \lambda_1\deg(P)+\lambda_2\log H(P)\,,
\]
and
\begin{equation}\label{eq:mino_intersec_ideals2}
\log |\I(\omeg_0)|
\le
\tfrac12 \log |\prescript{\rm h}{}{P}(\omeg_0)|
+\lambda_1\deg(P)+\lambda_2\log H(P)\,.
\end{equation}
After enlarging $\lambda_1$ if necessary, we assume that $\lambda_1\ge D_0$, where $D_0$ is the constant given by Lemma~\ref{prop:majo_ideal_bis}. Since $D=\lambda_1\deg(P)$, 
it follows that $D\ge D_0$. Recall that $\hgt(\I)=m+1-t$, and hence
\[
r:=m+1-\hgt(\I)=t\,.
\]

We now assume that
\[
\lambda_1\deg(P)+\lambda_2\log H(P)\leq (\lambda_1\deg(P))^{2t+3}\,,
\]
since the complementary case was already treated in the proof of
Theorem~\ref{thm:measure_1st_regime}. Set $d:=\deg(P)$ and
\[
\delta:=
\max\left\{t+2,\; \dfrac{\log(\lambda_1 d+\lambda_2\log H(P))}{\log(\lambda_1 d)}\right\}
\]
Then
\[
t+2\leq \delta \leq 2t+3\,,
\]
and therefore $\I$ satisfies the assumptions of Lemma~\ref{prop:majo_ideal_bis} with $r=t$.

Applying \eqref{eq:mino_Iomega_bis} together with \eqref{eq:mino_intersec_ideals2}, and using the fact that $\delta\geq 1$, we obtain
\begin{align*}
\log |\I(\omeg_0)|
&\geq -\lambda \left(D^t\log H(\I)+D^{t-1+\delta}\deg(\I)\right)\\
&\geq -\lambda_3\bigl(d^t\log H(P)+d^{t+\delta}\bigr)\,.
\end{align*}

We now distinguish two cases.

If $\delta = t+2$, then
\[
\log |\I(\omeg_0)| \geq -\lambda_4\bigl(d^t\log H(P)+d^{2t+2}\bigr)\,.
\]
This proves the theorem in that case.

Otherwise,
\[
\delta=
\dfrac{\log(\lambda_1 d+\lambda_2\log H(P))}{\log(\lambda_1 d)}\,,
\]
so that
\[
d^\delta=\mathcal O\bigl(d+\log H(P)\bigr)\,.
\]
Therefore,
\[
\log |\I(\omeg_0)| \geq -\lambda_5\bigl(d^t\log H(P)+d^{t+1}\bigr)\,.
\]
This completes the proof of Theorem~\ref{thm:measure_system}.
\end{proof}

\section{Proof of Theorem~\ref{thm:measure_sans_system}}\label{sec:remove_sing}

To deduce Theorem~\ref{thm:measure_sans_system} from Theorem~\ref{thm:measure_system}, we need to identify a suitable Mahler system involving the functions $f_1,\ldots,f_r$, together with possibly some additional functions, such that $\alpha$ is regular for this system. The following result provides the necessary background for applying Theorem~\ref{thm:measure_system}.


\begin{prop}\label{prop:good_eq}
Let $f$ be an $M_q$-function, and let $\alpha \in \Q^\times$ with $|\alpha|<1$.
Assume that $f$ is well-defined at $\alpha$, and let $\ell$ be a positive integer such that
\[
|\alpha^{q^\ell}|<\rho\,,
\]
where $\rho>0$ denotes the radius of convergence of $f$. Then $f$ satisfies a $q^\ell$-Mahler equation for which $\alpha$ is a regular point.
\end{prop}

\begin{proof}
This follows from the proof of \cite[Prop.~2.5]{AF24_EM}. That proposition shows that an $M_q$-function which is analytic in a disk centered at the origin of radius $R<1$ satisfies a $q$-Mahler equation with no singularities in this disk. A minor modification of the proof yields the following statement: if $f$ is 
well-defined at $\alpha^{q^k}$ for every integer $k\geq 0$, then $f$ satisfies a $q$-Mahler equation for which $\alpha$ is regular.

Now $f$ is also an $M_{q^\ell}$-function, and our assumption on $\ell$ implies that
\[
|\alpha^{(q^\ell)^k}|<\rho \qquad \text{for all } k\geq 1\,.
\]
Since $f$ is well-defined at $\alpha$ by assumption, it follows that $f$ is well-defined at $\alpha^{(q^\ell)^k}$ for every $k\geq 0$. Applying the previous statement with $q$ replaced by $q^\ell$, we obtain a $q^\ell$-Mahler equation satisfied by $f$ for which $\alpha$ is a regular point.
\end{proof}

\begin{rem}
The proof of Proposition~\ref{prop:good_eq} ultimately relies on the lifting theorem for $M_q$-functions \cite[Thm.~1.4]{AF17}; see also \cite[Thm.~1.3]{Ph15}. Thus, unlike Theorem~\ref{thm:measure_system}, Theorem~\ref{thm:measure_sans_system} depends on relatively recent developments in Mahler's method.
\end{rem}

We are now ready to prove Theorem~\ref{thm:measure_sans_system}.

\begin{proof}[Proof of Theorem~\ref{thm:measure_sans_system}]
Let $\ell$ be a sufficiently large integer so that $|\alpha^{q^\ell}|$ is smaller than the radius of convergence of each of the functions $f_i$, $1\le i\le r$. By Proposition~\ref{prop:good_eq}, each $f_i$ satisfies a $q^\ell$-Mahler equation
\begin{equation}\label{eq:goodM}
a_{i,0}(z)f_i(z)+a_{i,1}(z)f_i(z^{q^\ell})+\cdots+a_{i,m_i}(z)f_i(z^{q^{\ell m_i}})=0\,,
\end{equation}
for which $\alpha$ is a regular point.

It follows that the column vector with coordinates
\[
f_i(z),\ f_i(z^{q^\ell}),\ \ldots,\ f_i(z^{q^{\ell(m_i-1)}})
\]
is a solution of the $q^\ell$-Mahler system
\[
Y(z^{q^\ell})=A_i(z)Y(z)\,,
\]
where
\begin{equation}\label{eq:Ai}
A_i(z)=
\begin{pmatrix}
0 & 1 & 0 & \cdots & 0 \\
\vdots & \ddots & \ddots & & \vdots \\
\vdots & & \ddots & \ddots & 0 \\
0 & \cdots & \cdots & 0 & 1 \\
-\dfrac{a_{i,0}(z)}{a_{i,m_i}(z)} &
\cdots &
\cdots &
\cdots &
-\dfrac{a_{i,m_i-1}(z)}{a_{i,m_i}(z)}
\end{pmatrix}\,.
\end{equation}
Moreover, $\alpha$ is regular for this system.

Therefore, the functions
\[
f_i(z^{q^{\ell k}})\,,
\qquad
1\le i\le r,\quad 0\le k\le m_i-1\,,
\]
form the coordinates of a solution vector of the $q^\ell$-Mahler system
\[
Y(z^{q^\ell})=A(z)Y(z)\,,
\]
where
\[
A(z)=A_1(z)\oplus\cdots\oplus A_r(z)=
\begin{pmatrix}
A_1(z) & & \\
& \ddots & \\
& & A_r(z)
\end{pmatrix}\,.
\]
The point $\alpha$ remains regular for this system. Furthermore, the transcendence degree over $\Q(z)$ of the field generated by all these functions is at most $\tau$. The conclusion now follows directly from Theorem~\ref{thm:measure_system}.
\end{proof}


\section{Mahler's classification}\label{sec: mahler}

Given a complex number $\xi$ and a positive integer $d$, let $w_d(\xi)$ denote the supremum of the real numbers $w$ for which the inequality
\[
0<|P(\xi)|<H(P)^{-w}
\]
holds for infinitely many polynomials $P\in \Z[X]$ of degree at most $d$.
Then $\xi$ is a Liouville number if and only if $w_1(\xi)=+\infty$.

We further define
\[
w(\xi):=\limsup_{d\to\infty}\frac{w_d(\xi)}{d}\,\cdot
\]
The complex numbers are divided into the following four classes:
\begin{itemize}
\item $\xi$ is an $A$-number if $w(\xi)=0$.
\item $\xi$ is an $S$-number if $0<w(\xi)<\infty$.
\item $\xi$ is a $T$-number if $w(\xi)=\infty$ and $w_d(\xi)<\infty$ for all $d\ge 1$.
\item $\xi$ is a $U$-number if $w_d(\xi)=\infty$ for some $d\ge 1$.
\end{itemize}
The $A$-numbers are precisely the algebraic numbers, while the $U$-numbers generalize Liouville numbers. Almost all complex numbers, in the sense of Lebesgue measure, are $S$-numbers. Distinguishing between $S$- and $T$-numbers is notoriously difficult in transcendental number theory.

In view of this classification, a natural heuristic is that once the transcendence of a complex number $\xi$ has been established, one expects $\xi$ to be an $S$-number, and in particular neither a Liouville number nor a $U$-number, unless there is specific evidence to the contrary. This is precisely the situation for elements of the set $\mathbf M$.

\begin{conj}\label{conj}
All transcendental elements of $\mathbf M$ are $S$-numbers.
\end{conj}

Conjecture~\ref{conj} seems to belong to the folklore of the subject, although it is difficult to trace a precise source. The only known case concerns values of $M$-functions satisfying a homogeneous or inhomogeneous Mahler equation of order~$1$, which was proved by Galochkin \cite{Gal80} in 1980.  

Subsequent results point in the same direction, but fall short of establishing Conjecture~\ref{conj} in full generality. They show that the number under consideration must be either an $S$-number or a $T$-number.  
In his correspondence with Shallit, Becker conjectured that transcendental automatic real numbers should be $S$-numbers. Using the Schmidt Subspace Theorem, Adamczewski and Bugeaud \cite{AB11} proved in 2011 a transcendence measure for irrational automatic real numbers, from which it follows that these numbers are either $S$-numbers or $T$-numbers. 
In 2015, Bell, Bugeaud, and Coons \cite{BBC} proved that values at points of the form $1/b$ of generating functions of regular sequences are also either $S$-numbers or $T$-numbers. Since generating functions of regular sequences form a broader class than those associated with automatic sequences, this gives a generalization of the previous result.

In the direction of Conjecture~\ref{conj}, we easily deduce from Theorem~\ref{thm:measure_sans_system} the following result, which generalizes Theorem~\ref{thm:Lnumbers}.

\begin{coro}
	\label{coro:U_numbers}
	No $\xi \in \bf M$ is a $U$-number.
\end{coro}

\begin{proof}
If $\xi$ is algebraic, then $\xi$ is not a $U$-number, and there is nothing to prove. Assume that $\xi$ is transcendental. Let $f(z)$ be an $M_q$-function, and let $\alpha$ be a nonzero algebraic number such that $f(\alpha)$ is well-defined and
\[
f(\alpha)=\xi\,.
\]
Then $f$ is transcendental and therefore $|\alpha|<1$. Set
\[
\tau = {\rm tr.deg}_{\Q(z)}\bigl(f(z^{q^\ell}) : \ell \geq 0\bigr)\,.
\]
Since $f$ is transcendental, we have $\tau \geq 1$.

For every positive integer $d$, Theorem~\ref{thm:measure_sans_system} with $r=1$ implies that there exists a positive real number $C_1$ such that
\[
|P(f(\alpha))| > H(P)^{-C_1d^{\tau}}e^{-C_1d^{2\tau+2}}
\]
for every nonzero polynomial $P \in \Z[X]$ of degree at most $d$. It follows that for every $\varepsilon>0$ there are only finitely many polynomials $P\in\Z[X]$ of degree at most $d$ such that
\[
|P(f(\alpha))| < H(P)^{-C_1d^\tau-\varepsilon}\,.
\]
Hence
\[
w_d(f(\alpha)) \leq C_1 d^\tau < +\infty\,.
\]
Therefore $\xi=f(\alpha)$ is not a $U$-number.
\end{proof}

When $\tau=1$ in Theorem~\ref{thm:measure_sans_system}, we conclude that $f(\alpha)$ is an $S$-number. This slightly generalizes the result of Galochkin mentioned above.

In view of Corollary~\ref{coro:U_numbers}, the remaining step toward Conjecture~\ref{conj} is to prove that no element of $\mathbf M$ is a $T$-number.

\section{Concluding remarks}\label{sec: last}

We end the paper with a few remarks.

\subsection{An application: transcendence over $M$-values}

One of the most celebrated applications of Liouville's inequality is the first proof of transcendence in history: the proof that the number
\[
\xi := \sum_{n=0}^\infty 10^{-n!}
\]
is transcendental. The argument proceeds by contradiction. Suppose that $\xi$ is algebraic. Then Liouville's inequality implies that there exists an integer $N$ such that, for every rational number $p/q$, one has
\begin{equation}\label{eq:lower_bund_Liouville}
\left| \xi - \frac{p}{q} \right| \geq q^{-N}\,.
\end{equation}
To derive a contradiction, consider the rational number
\[
\frac{p}{q}=\sum_{n=0}^{N}10^{-n!}\,,
\qquad\text{so that}\qquad
q=10^{N!}\,.
\]
Then
\[
\left| \xi - \frac{p}{q} \right|
\leq \sum_{n=N+1}^\infty 10^{-n!}
\leq 2\cdot 10^{-(N+1)!}
\leq 2q^{-(N+1)}
< q^{-N}\,,
\]
contradicting \eqref{eq:lower_bund_Liouville}. This proves that $\xi$ is transcendental.

A similar strategy can be adapted to the ring
$\mathbf M_{q,\alpha}$ defined in \eqref{eq:mqalpha}, in place of $\Q$. 
Recall that $\mathbf M_{q,\alpha}$ contains $\Q$, but is not expected to be a field. When $\alpha=b^{-1}$ is the reciprocal of an integer $b\geq 2$, the ring $\mathbf M_{q,\alpha}$ contains all real numbers whose base-$b$ expansion is generated by a $q$-automatic sequence; see Chapter~13 of \cite{AS03}.  
We now deduce from Theorem~\ref{thm:measure_sans_system} the following result \emph{\`a la Liouville}.

\begin{coro}
Let $\beta \in {\rm Frac}\left(\mathbf M_{q,\alpha}\right)$ with $|\beta|< 1$, and let $(u_n)_{n\ge 0}$ be an increasing sequence of positive integers such that
\[
\limsup_{n\to\infty}\frac{u_{n+1}}{u_n^C}=+\infty
\qquad\text{for every } C>0\,.
\]
Then the number
\[
\xi:=
\sum_{n\ge 0}\beta^{u_n}
\]
is transcendental over $\mathbf M_{q,\alpha}$. In particular, $\xi$ is transcendental.
\end{coro}

\begin{rem}  
The sequence $(u_n)_{n\ge 0}$ is required to grow faster than the sequence $n!$ in Liouville's classical construction. Any increasing sequence satisfying
\[
\limsup_{n\to\infty}\frac{\log\log u_n}{n}=+\infty
\]
meets this requirement. For instance, one may take $u_n=2^{n!}$.
\end{rem}

\begin{proof}
Assume, for contradiction, that $\xi$ is algebraic over ${\bf M}_{q,\alpha}$. 
Then there exist an integer $r\ge 3$, $M_q$-functions $f_1,\ldots,f_{r-2}$, and a non-zero homogeneous polynomial
\[
P\in \Z[X_1,\ldots,X_{r-2},Y]
\]
such that
\[
P(f_1(\alpha),\ldots,f_{r-2}(\alpha),\xi)=0\,,
\]
while
\[
P(f_1(\alpha),\ldots,f_{r-2}(\alpha),Y)\neq 0\,.
\]
Let $d=\deg(P)$, and let $f_{r-1},f_{r}$ be $M_q$-functions such that
$\beta=\frac{f_{r-1}(\alpha)}{f_{r}(\alpha)}$. 

For each integer $N\ge 0$, set
\[
\xi_N=\sum_{n=0}^{N}\beta^{u_n}\,.
\]
Then
\[
f_{r}(\alpha)^{u_N}\xi_N=\sum_{n=0}^{N}f_{r-1}(\alpha)^{u_n}f_{r}(\alpha)^{u_N-u_n}\,,
\]
so there exists a homogeneous polynomial $Q_N\in \Z[X_{r-1},X_{r}]$ of degree at most $u_N$ and height equal to \(1\) such that
\[
Q_N(f_{r-1}(\alpha),f_{r}(\alpha))=f_{r}(\alpha)^{u_N}\xi_N\,.
\]
Define
\[
P_N(X_1,\ldots,X_{r})
:=
P(X_1X_{r}^{u_N},\ldots,X_{r-2}X_{r}^{u_N},Q_N(X_{r-1},X_{r}))\, .
\]
Then
\[
P_N(f_1(\alpha),\ldots,f_{r}(\alpha))
=
f_{r}(\alpha)^{du_N}P(f_1(\alpha),\ldots,f_{r-2}(\alpha),\xi_N)\,.
\]
In particular, $P_N\in \Z[X_1,\ldots,X_{r}]$, its degree is at most \(d(u_N+1)\), and its height is bounded above by a constant \(H\) independent of \(N\).

Moreover, since
\[
P(f_1(\alpha),\ldots,f_{r-2}(\alpha),\xi)=0
\]
and
\[
P(f_1(\alpha),\ldots,f_{r-2}(\alpha),Y)\neq 0\,,
\]
there exists a constant \(\kappa_1>0\) such that
\begin{align*}
|P_N(f_1(\alpha),\ldots,f_{r}(\alpha))|
&=
|f_{r}(\alpha)|^{du_N}
\bigl|
P(f_1(\alpha),\ldots,f_{r-2}(\alpha),\xi_N) \\
&\hspace{3.5cm}
-
P(f_1(\alpha),\ldots,f_{r-2}(\alpha),\xi)
\bigr|
\\
&\le
\kappa_1 |f_{r}(\alpha)|^{du_N} |\xi-\xi_N|\,.
\end{align*}
Since \((u_n)\) is increasing and \(|\beta|<1\), there exists \(\kappa_2>0\) such that
\[
|\xi-\xi_N|
\le
\sum_{n\ge N+1}|\beta|^{u_n}
\le
\kappa_2 |\beta|^{u_{N+1}}\,,
\]
hence
\[
|P_N(f_1(\alpha),\ldots,f_{r}(\alpha))|
\le
e^{-\kappa_3 u_{N+1}}
\]
for some constant \(\kappa_3>0\).

On the other hand, since $P(f_1(\alpha),\ldots,f_{r-2}(\alpha),Y)\neq 0$, we have  \[P(f_1(\alpha),\ldots,f_{r-2}(\alpha),\xi_N)\neq 0\]  for all sufficiently large $N$. Therefore,   Theorem~\ref{thm:measure_sans_system} gives
\[
|P_N(f_1(\alpha),\ldots,f_{r}(\alpha))|
\ge
H^{-C_1\deg(P_N)^\tau}
e^{-C_1\deg(P_N)^{2\tau+2}}
\ge
e^{-\kappa_4 u_N^{2\tau+2}}\,,
\]
where \(\tau\) is the transcendence degree appearing in Theorem~\ref{thm:measure_sans_system}, and \(\kappa_4>0\) is independent of \(N\).

Taking logarithms, we obtain
\begin{equation}\label{eq:cond_uN}
\kappa_3 u_{N+1}\le \kappa_4 u_N^{2\tau+2}\,,
\end{equation}
for all \(N\) large enough, which contradicts the assumption
\[
\limsup_{N\to\infty}\frac{u_{N+1}}{u_N^{2\tau+2}}=+\infty\, .
\]
This contradiction proves the result.
\end{proof}

If one has additional information on $\beta$ and is only interested in transcendence over $\Q(\beta)$, then the growth condition on the sequence $(u_n)_{n\ge 0}$ can be weakened.

\begin{coro}\label{coro:serie_lac_2}
Let $\beta$, with $|\beta|<1$, be the ratio of the values at $\alpha\in\Q$ of two $M_q$-functions satisfying homogeneous or inhomogeneous Mahler equations of orders $m_1$ and $m_2$, respectively. Let $(u_n)_{n\ge 0}$ be an increasing sequence of positive integers such that
\[
\limsup_{n\to\infty}\frac{u_{n+1}}{u_n^{2(m_1+m_2)+2}}=+\infty\,.
\]
Then the number
\[
\xi:=\sum_{n\ge 0}\beta^{u_n}
\]
is transcendental over $\Q(\beta)$. In particular, $\xi$ is transcendental.
\end{coro}

\begin{proof}
Suppose, for contradiction, that $\xi$ is algebraic over $\Q(\beta)$. Write
\[
\beta=\frac{f_1(\alpha)}{f_2(\alpha)}\,,
\]
where, for $i\in\{1,2\}$, the function $f_i$ satisfies a homogeneous or inhomogeneous Mahler equation of order $m_i$. We argue exactly as in the proof of the previous corollary, taking $r=5$ and setting
\[
f_3=1\,,  f_4=f_1\,,\qquad f_4=f_2\,.
\]
In this case, the quantity $\tau$ is at most $m_1+m_2$; see Remark~\ref{rem:tau}. The assumption on the sequence $(u_n)_{n\ge 0}$ then implies that \eqref{eq:cond_uN} cannot hold for all sufficiently large $N$, yielding the desired contradiction.
\end{proof}

\begin{ex*}
Let $\beta$ be the binary Thue--Morse number, that is,
\[
\beta:=\sum_{n=0}^\infty \frac{t_n}{2^n}\,,
\]
where $t_n$ is the parity of the sum of the binary digits of the integer $n$. 

Then the numbers
\[
\xi:=\sum_{n=0}^\infty \beta^{2^{5^n}}
\qquad\text{and}\qquad
\beta
\]
are algebraically independent over $\Q$. In particular, $\xi$ is transcendental.

Indeed, we have $\beta=f(1/2)/g(1/2)$, where
\[
f(z)=\sum_{n=0}^\infty t_n z^n,\quad \text{ and }\quad g(z)=1\,.
\]
The generating function $f$ satisfies the inhomogeneous $2$-Mahler equation
\[
f(z)=(1-z)f(z^2)+\frac{z}{1-z^2}\,,
\]
which is of order $m=1$, while the constant function $g=1$ satisfies an inhomogeneous $2$-Mahler equation of order $0$.  

Now let
\[
u_n=2^{5^n}\,.
\]
Then
\[
\frac{u_{n+1}}{u_n^4}
=
2^{5^{n+1}-4\cdot 5^n}
=
2^{5^n}\xrightarrow[n\to\infty]{}\infty\,.
\]
Hence the growth condition in Corollary~\ref{coro:serie_lac_2} is satisfied, and it follows that $\xi$ is transcendental over $\Q(\beta)$. Therefore $\beta$ and $\xi$ are algebraically independent over $\Q$.
\end{ex*}

\subsection{Effectivity}  
In principle, the constant $C_1$ appearing in Theorems~\ref{thm:measure_system}
and~\ref{thm:measure_sans_system} could be made effective. Its value depends on
various parameters attached to the $M$-functions under consideration and to the
algebraic point at which they are evaluated.

However, some of the relevant parameters are governed by the structure of the ideal
\[
\p_z := \{P \in \C[z][X_0, \ldots, X_m] \text{ homogeneous in } \X :
P(z,1,f_1,\ldots,f_m)=0\}\,,
\]
which encodes the homogeneous algebraic relations over $\C[z]$ among
the functions $1,f_1,\ldots,f_m$. This is already the case for the constant occurring in the
multiplicity lemma (Lemma~\ref{lem:Multiplicity}).

Moreover, in the final part of the proof of Theorem~\ref{thm:measure_system}, we
introduce the ideal
\[
\p := \{Q \in \Z[\X] \text{ homogeneous} :
Q(1,f_1(\alpha),\ldots,f_m(\alpha))=0\}\,.
\]
The constant $C_1$ also depends on $\deg(\p)$ and $\log H(\p)$. By the lifting
theorem, homogeneous algebraic relations at the point $\alpha$ come from
homogeneous algebraic relations between the functions themselves, so that
\[
\p = {\rm ev}_{z=\alpha}(\p_z)\cap \Z[\X]\,.
\]
Consequently, any effective control of $\deg(\p)$ and $\log H(\p)$ ultimately depends on our understanding of $\p_z$.

Except in a few special cases, determining $\p_z$ is a difficult problem, closely related to the difference Galois theory of Mahler equations. Feng~\cite{Fe18} developed an algorithm for computing the Galois group of linear difference equations with rational function coefficients in the shift case. According to a personal communication from Feng, the underlying ideas could adapt to the Mahler setting. However, even in the shift case, the computational complexity of these procedures appears too high for most practical purposes. Consequently, we do not expect to obtain a reasonably effective version of our main result.


\end{document}